\documentclass[lettersize,journal]{IEEEtran}
\usepackage{amsmath,amsfonts}
\usepackage{algorithmic}
\usepackage{algorithm}
\usepackage{array}
\usepackage[caption=false,font=normalsize,labelfont=sf,textfont=sf]{subfig}
\usepackage{textcomp}
\usepackage{stfloats}
\usepackage{url}
\usepackage{verbatim}
\usepackage{graphicx}
\usepackage{cite}
\newtheorem{proposition}{Proposition}
\newtheorem{theorem}{Theorem}
\newtheorem{remark}{Remark}

\def\erf{\text{erf}}
\usepackage{balance}
\usepackage{hyperref}
\usepackage{academicons}
\usepackage{xcolor}
\usepackage{xspace}

\usepackage{tikz}
\newcommand\publishedtext{%
	\footnotesize \hspace{4cm}This is the authors' extended version of the paper submitted to IEEE Communications Letters.} 

\newcommand\copyrightnotice{%
\begin{tikzpicture}[remember picture,overlay]
\node[anchor=north,yshift=0pt] at (current page.north) {\fbox{\parbox{\dimexpr\textwidth-\fboxsep-\fboxrule\relax}{\publishedtext}}};
\end{tikzpicture}%
}
\begin{document}

\title{Power minimization and resource allocation in HetNets with uncertain channel-gains}

\author{Gabriel O. Ferreira\href{https://orcid.org/0000-0002-4592-8975}{\textsuperscript{\includegraphics[scale=0.01]{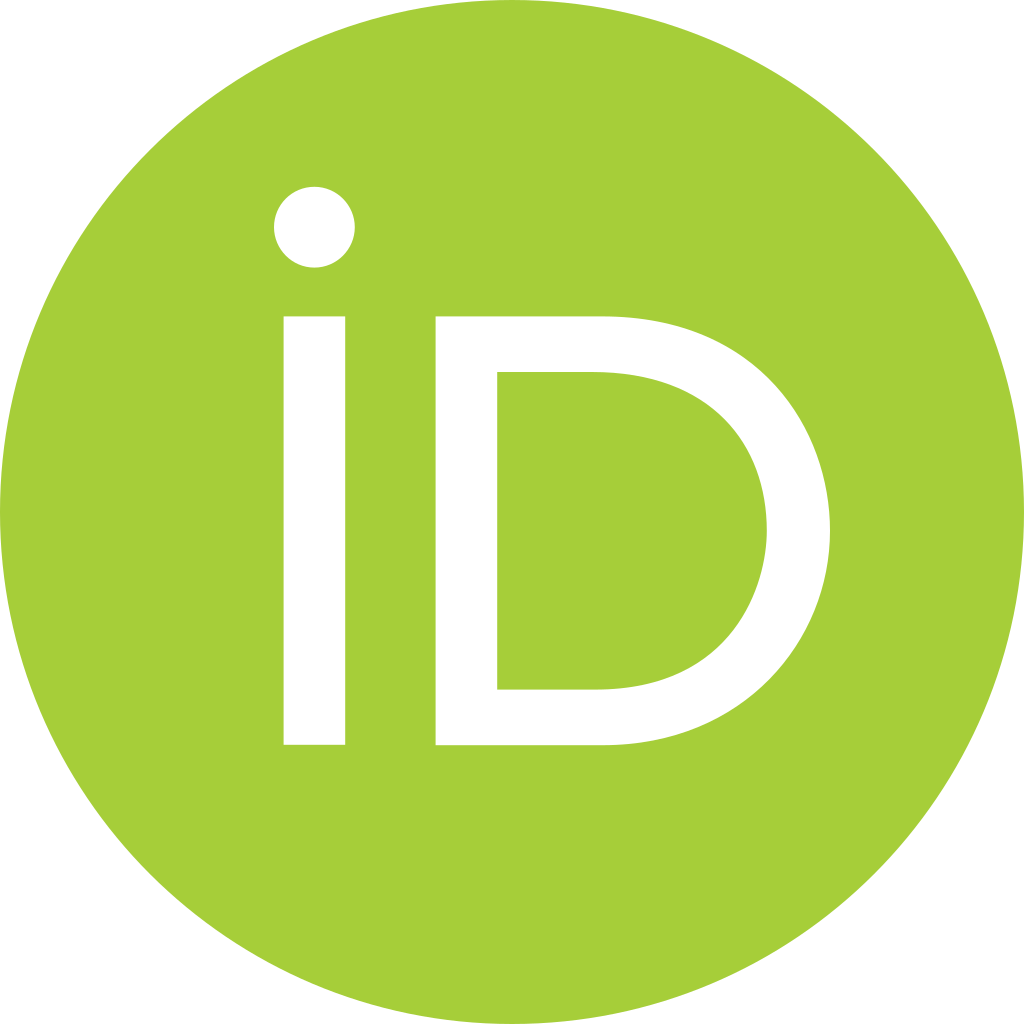}}},~\IEEEmembership{Member,~IEEE,} Chiara Ravazzi\href{https://orcid.org/0000-0002-3186-6195}{\textsuperscript{\includegraphics[scale=0.01]{Figures/ORCID_Logo.png}}},~\IEEEmembership{Member,~IEEE,} \\ Fabrizio Dabbene\href{https://orcid.org/0000-0002-2258-8971}{\textsuperscript{\includegraphics[scale=0.01]{Figures/ORCID_Logo.png}}},~\IEEEmembership{Senior Member,~IEEE,} Giuseppe C. Calafiore\href{https://orcid.org/0000-0002-6428-5653}{\textsuperscript{\includegraphics[scale=0.01]{Figures/ORCID_Logo.png}}},~\IEEEmembership{Fellow,~IEEE}
\thanks{Gabriel O. Ferreira, Chiara Ravazzi, and Fabrizio Dabbene are with IEIIT CNR Institute, Torino, Italy. (e-mail: gabriel.oliveiraferreira@ieiit.cnr.it; chiara.ravazzi@ieiit.cnr.it; fabrizio.dabbene@cnr.it). Giuseppe C. Calafiore is with DET, Politecnico di Torino, Italy, and CECS, VinUniversity, Vietnam (e-mail: giuseppe.calafiore@polito.it).}
\thanks{Work supported by PRIN project ``TECHIE: A control and network-based approach for fostering the adoption of new technologies in the ecological transition'' Cod. 2022KPHA24 CUP: B53D23002760006 and by the European Commission Horizon 2020 Framework Programme, H2020-MSCA-ITN-2019, MSCAITN-EID, Proposal No. 860239, BANYAN.}
}



\maketitle

\copyrightnotice\vspace*{-2pt}

\begin{abstract}
We propose an optimization problem to minimize the base stations transmission powers in OFDMA heterogeneous networks, while respecting users' individual throughput demands. The decision variables are the users' working bandwidths, their association, and the base stations transmission powers. To deal with wireless channel uncertainty, the channel gains are treated as random variables respecting a log-normal distribution,  leading to a non-convex chance constrained mixed-integer optimization problem, which is then formulated as a mixed-integer Robust Geometric Program. The efficacy of the proposed method is shown in a real-world scenario of a large European city. 
\end{abstract}


\section{Introduction}
\label{sec:Intro}

\IEEEPARstart{M}{obile} networks efficiency is a subject that has received significant attention from the scientific community as a consequence of the exponentially growing number of connected devices and their increasing traffic demands. Due to the limited Radio Access Networks (RAN) resources, a wide range of optimization problems such as spectral efficiency, base stations (BS) transmission powers design, and association between users and BS are of primary importance, considering that non-optimized RANs could lead to high energy consumption, paired with low data transmission capacity \cite{David2022}. 

Aiming at higher data rate for the users, fifth-generation of mobile cellular networks (5G) exploits frequency bands that can go from 5 GHz up to 300 GHz, leading to severe signal attenuation. In this context, network operators make use of heterogeneous networks (macro, micro, and femto cells deployed in relatively proximity) to provide seamless connection to the users' equipment (UE).

The channel gains between UEs and BSs have a fundamental importance when optimizing mobile networks efficiency. In general, these parameters can be evaluated by means of site specific ray tracing solvers that simulate the radio waves propagation. These strategies (or any other relying on physics-based models) have the limitation of depending on how accurate the environment is described and reconstructed into the propagation solvers/models. However, even if the buildings geometry and electromagnetic properties of the construction materials are known, parameters such as motion of objects (people, cars, etc) make these quantities stochastic in practice \cite{UQ_GlobeCom}. In real life situations the electromagnetic waves propagated by the antennas will suffer not only from shadowing, but also combined effects of reflection and diffraction. Consequently, the channel gain between a UE/BS pair is inherently a random quantity. Additionally, the log-normal distribution has been empirically proved to be an accurate model for the channel gains in indoor and outdoor environments \cite{Goldsmith_2005}.


In this paper, we extend the work presented in \cite{Ferreira_2024} to the case where the channel gains are random variables. We propose a joint optimization algorithm that minimizes the BSs transmission powers in OFDMA heterogeneous networks, while respecting chance-constrained individual users quality of service (QoS) requirements. Unlike many works in the literature (discussed in Section \ref{sec:RW}), our optimization problem does not rely on iterative/sequential procedures and does not require a known feasible initial solution, which in general leads to local optimal.

In Section \ref{sec:RW} we present the related works, some limitations of the techniques used to tackle to problem, and our contributions. The mathematical formulation of the problem is presented in Section \ref{sec:PF}. A Mixed-Integer Robust Geometric Program to solve the original Mixed-Integer non-convex chance constrained one is derived in Section \ref{sec:RobustMIGP}. Subsequently, numerical experiments are discussed in Section \ref{sec:example}. Lastly, conclusions and future works are presented. 

\textbf{Notation:} $[n]$ denotes $\{1,\ldots,n\}$. $C_{ij\ell}$ stands for a constant (or respective variable) obtained with parameters corresponding to user $i \in [n]$, base station $j \in [N]$, and approximation function $\ell \in [m]$. Similarly, $C_{ij}$, $C_{i}$, and $C_{j}$ denote constant/variable obtained with parameters related to: 1- user $i$ and base station $j$, 2- user $i$, and 3- base station $j$.

\section{Related work}
\label{sec:RW}

Different approaches are used to address the resource allocation problem. First, we present papers with distinct strategies to deal with the highly non-convex resource allocation characteristic. Subsequently, papers considering the channel-gains as random variables are introduced. We then conclude the section with the limitations of the cited approaches and highlight the contribution of the proposed solution.

\subsection{Resource allocation}

We refer to \textit{resource allocation} as the problem of optimizing users/BS association, bandwidth allocation for each user, and the BSs transmission powers. 

Due to the non-convex nature of resource allocation problems, some papers propose solutions based on successive approximation methods, for instance iteratively solving linear approximations of the original problem inside small regions until a local minimum or stop condition is achieved. Another common procedure is splitting the problem into smaller ones, where the decision variables are divided in subsets of constant and non constant parameters. Then a solution to a simpler version of the problem is achieved. Subsequently, the obtained solution becomes the subset of constant parameters and what was previously considered constant becomes the subset of variables. This procedure is repeated until a stop condition is achieved \cite{wang2017, Le2019}. In \cite{sun_2}, high-altitude platform with simultaneous wireless information and power transfer networks are studied. By using strategies such as discretization method, Log-Sum-Exp-dual scheme, and modified cyclic coordinate descent, the non-convex problem of maximizing a worst-case sum rate is then solved with an scalable robust optimization framework. In \cite{sun_1}, the authors propose a fast-convergence optimization for the problem of semantic computation rate under jamming attacks and channel state information imperfection. In summary, an algorithm based on monotonic optimization combined with second-order cone programming is developed to deal with the resulting quai-convex objective function and mixed-integer non-linear constraints.

Convex approaches usually circumvent the non-convexity either by setting lower/upper bounds to these functions or by approximating them with convex ones. Such technique has the advantage of not requiring a known feasible initial solution. Furthermore, these problems can be solved globally and efficiently. In \cite{Boyd:03}, a convex approach based on Geometric Programming (GP) is formulated to address simultaneously the problem of power allocation and routing in code-division multiple access wireless data networks. In \cite{Bao2015} resource allocation in heterogeneous networks is studied. More specifically, the decision variables are the radio resources, supporting elastic and inelastic network traffic, across multiple tiers and the objective function to be maximized is the downlink sum throughput, which is non-convex. The authors derive concave upper and lower bounds and solve an approximate convex problem. However, BSs transmission powers and users' association are not problem variables.
For other approaches applied to resource allocation in 5G heterogeneous networks, the reader may refer to the Survey in~\cite{xu2021}.


\subsection{Stochastic channel-gains}

The work of \cite{Boyd2002} presents power control optimization problems based on GP where the received power at a user equipment is an exponentially distributed random variable. Different problem derivations, such as minimize outage probability or transmission powers subject to a maximum outage probability, are studied. Even though GPs are convex optimization problems, the authors also present a fast heuristic approach to solve the minimization of outage probability. In \cite{Boyd2005}, power control in log-normal fading wireless channels is studied. The authors show that the problem of finding an optimal solution while respecting the users' quality of service can be formulated as a Stochastic GP. However, such a problem is very hard to solve and a relaxation based on Robust GP is proposed. 

\subsection{Limitations of literature approaches and our contributions}

Solutions based on successive approximation methods, such as gradient-descent based or difference of two convex functions (DC programming), usually have the limitation of being highly dependent on the initial feasible point provided to them, leading to distinct optimal values. More importantly, for large scale optimization problems, it might be extremely difficult  to find a feasible point where to start the iterative process. Additionally, the approximations only work inside a validity region close enough to the point of approximation. This might lead to time consuming solutions since the steps towards a local minimum inside the decision variables set are usually small, requiring possibly many iterations. NNs approaches require considerably amount of training data, which might not be available in real-world applications. 

This work lies on the convex approaches field, therefore the global optimum is guaranteed, there is no need of \textit{a priori} knowledge about a feasible initial condition, and no training data is required. Additionally, we also present contributions with respect to the approaches that also tackle the problem from a convex optimization perspective, such as: 1-the users throughput levels also depend on the amount of bandwidth associated to each one, and our approach includes decision variables to allocate the resources in an optimized manner; 2- with a suitable change of variables followed by a piecewise power function approximation, we formulate a less conservative lower bound to the \textit{Shannon-Hartley Theorem}.



\section{Problem Formulation}
\label{sec:PF}

Consider the following problem variables and constants:
\begin{itemize}

    \item[i)] $N$ (constant): number of base stations in an OFDMA heterogeneous network.

    \item[ii)] $n$ (constant): number of users that need to be associated to one base station among $N$ possible ones.
    
    \item[iii)] $P_j$ (variable): transmission power (W) of one resource block (RB) in BS $j$. Let $\mathcal{P}  \doteq [P_1,\ldots,P_N]$.

    \item[iv)] $\hat{P_{j}}$ (constant): maximum physical limit on the transmission powers of one resource block of base station $j$.
    
    \item[v)] $r_{i}$ (constant): throughput requirement, measured in (bits/s), of user $i$.

    \item[vi)] $B_j$ (constant): bandwidth (Hz) of BS $j$.
    
    \item[vii)] $x_{ij}$ (variable): resources of BS $j$ assigned to user $i$. Hence, $x_{ij}{B_{j}}$ can be seen as the working bandwidth assigned by BS $j$ to user $i$. We denote by $x\in [0,1]^{n\times N}$ the matrix of the elements $x_{ij}$, with $x_{ij} \in [0,1]$. 

    \item[viii)]$g_{ij}$: channel gain between user $i$ and BS $j$. They are random variables with known probability distribution, as it will be further addressed.

    \item[ix)]$z_{ij}$ (variable): binary variable $z_{ij}=1$ if a user $i$ is connected to BS $j$ and $z_{ij}=0$ otherwise. $\bar{z}_{ij} = 1-z_{ij}$.

    \item[x)] $M$ (constant): sufficiently large constant to apply the Big-M method. 

    \item[xi)] $\eta^{2}$(constant): noise power.
    
    \item[xii)] $S(\mathcal{P})$: Signal-to-Interference-Noise-Ratio. It is the ratio of the power of the wanted signal to the sum of noise and interfering powers of the other BSs, defined as:
    \begin{equation}\label{SINR}
       S_{ij}(\mathcal{P}) = \frac{P_jg_{ij}}{\eta^{2} + \sum_{k \neq j}P_k g_{ik}}, 
    \end{equation}
 where $\mathcal{P} = [P_1, \ldots, P_N ]$.
\end{itemize}

In \cite{Ferreira_2024}, we propose the following Mixed-Integer Geometric Programming (MIGP)
\begin{subequations}
\label{eq:MIGP}
\begin{align} 
 \min_{\bar z_{ij}, u_{ij}, q_j} \quad & \sum_{j=1}^{N} e^{q_j} \tag{MIGP} \label{prob_MIGP}\\
\textrm{s.t.:} & \nonumber\\
  & e^{u_{ij}} \leq 1, \quad  i\in[n], j\in[N]\label{eq:1a}\\
  &\,\bar z_{ij}\in\{0,1\}, \quad i\in[n], \quad j\in[N],\\
  & e^{q_j}  \leq \hat{P_{j}}, \quad j\in[N],\\
  & \sum_{j=1}^N \bar z_{ij} = N-1, \quad i\in[n],\\
  & \sum_{i=1}^{n} e^{u_{ij}} \leq 1, \quad  j\in[N],\label{eq:1e}\\ 
 & \hat{f_{\ell}}(q_j,u_{ij}) \leq A_{ij}^\ell + M \bar z_{ij}, \label{eq:1f}\\
   & \hspace{2.5cm}i\in[n], j\in[N], \ell\in[m],\nonumber
\end{align}
\end{subequations}
with
\begin{align}
\label{eq:PCA}
    \hat{f_{\ell}}(q_j,u_{ij}) &\doteq \log \left(\frac{\eta^{2}}{g_{ij}}e^{-q_j-\frac{u_{ij}}{b_{\ell}}}+\sum_{k \neq j} \frac{g_{ik}}{g_{ij}}e^{q_k-q_j-\frac{u_{ij}}{b_{\ell}}}\right),\\
    A_{ij\ell}&\doteq \frac{\log\left(\frac{B_ja_{\ell}}{r_i}\right)}{b_{\ell}}.
\end{align}
It can be shown (see \cite{Ferreira_2024} for details) that MIGP is an upper-bound solution to the non-convex integer optimization problem \eqref{nonconvex-OFDMA}, whose goal is to minimize transmission powers of base stations while respecting individual users' throughput constraints. 
\begin{subequations}
\label{min_NC}
    \begin{align}
        & \min_{x_{ij}, z_{ij}, P_j} \hspace{0.5cm} \sum_{j=1}^{N} P_j \tag{non-convex-OFDMA}\label{nonconvex-OFDMA}\\
        \textrm{s.t.:} & \nonumber\\
        & \,x_{ij}\in[0,1], \hspace{0.47cm} i\in[n], \quad j\in[N],\\
        &\,z_{ij}\in\{0,1\}, \quad i\in[n], \quad j\in[N],\\
        &\,0 \leq P_j \leq \hat{P_{j}}, \quad \quad j\in[N], \label{eq:c3}\\
        & \sum_{j=1}^{N} z_{ij} = 1, \quad i\in[n],\label{eq:c4}\\
        & \sum_{i=1}^{n} x_{ij} \leq 1, \quad j\in[N],\label{eq:c5}\\
        & 
\sum_{j=1}^N \left[x_{ij} B_j \log_2(1+S_{ij}(\mathcal{P}))\right] z_{ij} \geq r_i, \quad i\in[n].\label{eq:c6}
    \end{align}
\end{subequations}

For an OFDMA heterogeneous network composed of $N$ BS and $n$ users, constraints \eqref{eq:c4}, \eqref{eq:c5}, \eqref{eq:c6}  ensure that each user is connected to just one BS, a BS cannot provide more resources than available, and each user has an individual minimum throughput level to be respected, respectively. 

Note that Geometric Programs are convex, therefore when the binary variables $z_{ij}$ are known, optimization problem \eqref{prob_MIGP} can be solved globally. \eqref{prob_MIGP} was obtained by applying the change of variables $P_j = e^{q_j}$ and $x_{ij} = e^{u_{ij}}$ followed by a piecewise power function approximation of the \textit{Shannon-Hartley Theorem}
\begin{equation*}
 \log_2(1+S_{ij}(\mathcal{P})) \simeq a_{\ell}(1+S_{ij}(\mathcal{P}))^{b_{\ell}}, \quad \ell\in[m],
\end{equation*}
where $a_{\ell} > 0 $ and $b_{\ell} \in (0,1)$ are \textit{a priori} calculated parameters of the $m$ functions used in the approximation. The reader is referred to \cite{Ferreira_2024} for additional details and for the related proofs. Note that in that work,  channel gains are considered to be deterministic.

\subsection{The channel gains as random quantities}

In real-world environments, the channel gains are unknown random quantities, due to the combined effects of shadowing, multipath, diffraction, etc. In particular, as shown by \cite{Goldsmith_2005}, these quantities respect a log-normal distribution, i.e., 
 the dB gains defined as 
\begin{equation}
    g_{ij}^{(dB)} \doteq 10 \log_{10} g_{ij}, \quad i\in[n], j\in[N]
\end{equation}
obey a Gaussian distribution $g_{ij}^{(dB)} \sim \mathcal{N} (\tilde{\mu}_{ij}, \tilde{\sigma}_{ij}^2)$. We write this as
$g_{ij}\sim\mathcal{LN} (\mu_{ij}, \sigma_{ij}^2)$. 
Furthermore, we can define the normalized gains as
\begin{equation}\label{eq:rho_ij}
    \rho_{ij}  \doteq \frac{10 \log_{10} g_{ij} - \tilde{\mu}_{ij}}{\tilde{\sigma}_{ij}} \sim \mathcal{N}(0, 1), \quad i\in[n], j\in[N].
\end{equation}
From \eqref{eq:rho_ij} and with $c = \frac{\log 10}{10}$, we have
\begin{equation}\label{eq_g_exp}
    g_{ij} = e^{c(\tilde{\mu}_{ij}+\rho_{ij} \tilde{\sigma}_{ij})}, \quad i\in[n], j\in[N].
\end{equation}

\subsection{Stochastic MIGP}
Since the channel gains are random variables and the users' throughput levels are highly dependent on them, constraint \eqref{eq:1f} becomes stochastic, and must be respected with a probability $\mathbb{P}_i$ given by $1-\alpha_i$, leading to the following chance-constrained optimization problem 
\begin{subequations}\label{min_SGP}
\begin{align}
 \min_{\bar z_{ij}, u_{ij}, q_j} \quad & \sum_{j=1}^{N} e^{q_j}\tag{CC-MIGP}\label{prob:CC-MIGP}\\
\text{s.t.: } & e^{u_{ij}} \leq 1, \quad  i\in[n], j\in[N]\\
  &\,\bar z_{ij}\in\{0,1\}, \quad i\in[n], \quad j\in[N],\\
  & e^{q_j}  \leq \hat{P_{j}}, \quad j\in[N],\\
  & \sum_{j=1}^N \bar z_{ij} = N-1, \quad i\in[n],\\
  & \sum_{i=1}^{n} e^{u_{ij}} \leq 1, \quad  j\in[N],\\
  & \mathbb{P}_i\left(\hat{f_{\ell}}(q_j,u_{ij}) \leq A_{ij\ell} + M \bar z_{ij}\right) \geq 1-\alpha_i \nonumber\\
   & \hspace{2cm} \ell\in[m], i\in[n], j\in[N].
\end{align}
\end{subequations}

Chance constrained optimization problems might be difficult to handle. Then, inspired by the work of \cite{Boyd2005} we present a robust formulation based solution.

\section{Robust MIGP}
\label{sec:RobustMIGP}

\begin{proposition}[Box uncertainty]
Let $\rho_{ij}$, $i\in[n]$, $j\in[N]$ be Gaussian random variables with zero mean and unitary standard deviation, as in \eqref{eq:rho_ij}. For each user, compute the quantities $[\underline{\rho}_{ij},\overline{\rho}_{ij}]$ such that
\begin{equation}
\begin{cases}
\displaystyle \mathbb{P}(\underline{\rho}_{ij} \leq {\rho}_{ij} \leq \overline{\rho}_{ij}) = \frac{\erf \left[\frac{\overline{\rho}_{ij}}{\sqrt{2}}\right]-\erf\left[\frac{\underline{\rho}_{ij}}{\sqrt{2}}\right]}{2} = \varphi_j, \hspace{0.2cm} j\in[N],\\
\hfill \\
\prod_{j=1}^N \varphi_j = 1-\alpha_i.\
\end{cases}
\end{equation}
Then a robust optimization problem considering the uncertainty box can be derived such that its optimal solution is feasible to the CC-MIGP.
\end{proposition}

\begin{theorem}
For uncertainty boxes given by Proposition 1, the optimal solution $u_{ij}^{\star}, \bar z_{ij}^{\star}, q_j^{\star}$ to 
\begin{subequations}\label{min_RGP}
\begin{align}
 \min_{\bar z_{ij}, u_{ij}, q_j} \quad & \sum_{j=1}^{N} e^{q_j}\tag{Robust-MIGP} \label{min_RGP_label}\\
\text{s.t.: } &  \eqref{eq:1a}-\eqref{eq:1e}\\
  & \log(\eta^{2} e^{\beta_{ij\ell}}+ \sum_{k \neq j} e^{\zeta_{ik\ell}}) \leq A_{ij\ell} + M \bar z_{ij},\nonumber\\
  &\hspace{0.3cm} \forall \rho_{ij} \in [\underline{\rho}_{ij},\overline{\rho}_{ij}] \hspace{0.2cm}\text{and} \hspace{0.2cm} \forall \rho_{ik} \in [\underline{\rho}_{ik},\overline{\rho}_{ik}] \nonumber\\  
   & \hspace{2cm}\quad \ell\in[m],i\in[n], j\in[N], \label{eq:Rob_C} 
\end{align}
\end{subequations}
with
\begin{align*}
    \beta_{ij\ell} &=-c(\tilde{\mu}_{ij}+\underline{\rho}_{ij}\tilde{\sigma}_{ij})-q_j-\frac{u_{ij}}{b_{\ell}},\\ 
    \zeta_{ik\ell} &= c(\tilde{\mu}_{ik}+\overline{\rho}_{ik}\tilde{\sigma}_{ik}-\tilde{\mu}_{ij}-\underline{\rho}_{ij}\tilde{\sigma}_{ij})+q_k-q_j-\frac{u_{ij}}{b_{\ell}},
    \end{align*}
is always feasible to the chance constrained optimization problem in \eqref{min_SGP}, i.e.
\begin{equation*}
\text{Robust MIGP}(u_{ij}^{\star}, x_{ij}^{\star}, P_j^{\star}) \implies \text{Stochastic MIGP}.   
\end{equation*}
\end{theorem}

\begin{proof}
Constraint \eqref{eq:Rob_C} is obtained by replacing the channel gains as given in \eqref{eq_g_exp} into \eqref{eq:PCA}. Additionally, the terms $\beta_{ij\ell}$ and $\zeta_{ik\ell}$ consider the worst-case scenario inside the uncertainty box, i.e. the problem is solved for $\underline{\rho}_{ij}$ (lowest channel gain) and $\overline{\rho}_{ik}$ (highest interference) for each user. Therefore, the optimal solution $u_{ij}^{\star}, x_{ij}^{\star}, P_j^{\star}$ is feasible for any other possible combination of $\rho_{ij} \in  [\underline{\rho}_{ij},\overline{\rho}_{ij}]$ and $ \rho_{ik} \in [\underline{\rho}_{ik},\overline{\rho}_{ik}]$ $\forall k \neq j$ inside the uncertainty box, concluding the proof. 
\end{proof}

Note that we can easily retrieve $ \hat{f_{\ell}}(q_j,u_{ij})$ when the channel gains are given constants by making $\tilde{\sigma}_{ij} = \tilde{\sigma}_{ik} = 0$. 

\begin{remark}
In many real-world networks, association between UEs and BSs are given. For our application, this means that the values of all $z_{ij}$ are known and the optimization problem \eqref{min_RGP} becomes a standard Geometric Program, for which a global optimal solution can be always found in polynomial-time. Regarding scenarios with unknown binary variables, a solution can be obtained with branch \& bound methods, which increases the computational effort exponentially with the problem size \cite{Boyd2007}. 
\end{remark}


\section{Application use cases}
\label{sec:example}
 
\subsection{Channel gains expected values}
\label{sec:CG}



We estimate the channel gains from a neighbourhood in a large European city by reconstructing it (considering building height, city layout, BSs location, etc) into a 3D ray tracing propagation software, allowing a careful quantification of both path-loss and multi-path components incorporating them into the respective expected value for the channel gains. Figure \ref{fig_std_cg}.(a) shows the channel gains expected values $\tilde{\mu}_{ij}$ and Figures \ref{fig_std_cg}.(b) and \ref{fig_std_cg}.(c) expose how the randomness in these quantities can lead to a different quality of signal and coverage. Therefore, a scenario optimized considering only the expected values might not respect all constraints in real-life applications. Specifically to the optimization problem at hand, not treating the channel gains as random quantities might lead to a solution where many UEs have less throughput than their respective required threshold, estimated according to \cite{Andre_2023}.

\subsection{Optimization Scenario}
\label{sec:os}

All results obtained in the following subsections were obtained considering a neighbourhood from a large European city. The scenario consists of $N=5$ base stations, $n=130$ users, the channel gains expected values are the ones presented in Figure \ref{fig_std_cg} (a), and $M = 10^6$. Additionally, we use $m=5$ functions in the piecewise power function approximation, whose parameters are
\begin{equation*}
\begin{split}
    a_\ell & = [1.4080, 0.7720, 1.3436, 2.0641, 2.8584] \quad \text{and}\\
    b_\ell & = [1, 0.7994 0.3928 0.2538 0.1840].
\end{split}
\end{equation*}
%
\subsection{On the probability of the chance constrained optimization}
\label{sec:CCOP}

The channel gains standard deviation give information on how disperse these values are with respect to the expected ones obtained with procedure in Section \ref{sec:CG}. Therefore, increasing $\tilde{\sigma}$ means that the robust approach needs to be solved for worse cases given a fixed probability in the constraint. The impact of this parameter in the optimal value is shown in Figure \ref{fig_sigma_rho}.(a) (here we used $\tilde{\sigma}_{ij} = \tilde{\sigma}$ $i\in[n], j\in[N]$). Note that $\tilde{\sigma} = 5$ and probability levels of 85\% and 90\% lead to infeasible problems.

We can fix the standard deviation of the channel gains and vary the probability level in the constraints by modifying the uncertainty intervals given by $\rho_{ij} \in [\underline{\rho}_{ij},\overline{\rho}_{ij}]$ and $\rho_{ik} \in [\underline{\rho}_{ik},\overline{\rho}_{ik}]$. Similarly to the previous case, when we increase these intervals (and the box of uncertainties), the robust optimization problem needs to respect these constraints for worse scenarios, degrading the optimal value as in Figure \ref{fig_sigma_rho}.(b). The problem becomes infeasible considering a probability level of 95\% and $\tilde{\sigma}=4$.
In another experiment, optimization problem \eqref{min_RGP} is solved by considering the uncertainty box intervals 
\begin{equation*}
\begin{split}
[\underline{\rho}_{ij},\overline{\rho}_{ij}] = [-2.04, \infty), \quad 
[\underline{\rho}_{ik},\overline{\rho}_{ik}] = (-\infty,2.04],
\end{split}
\end{equation*}
and standard deviation $\tilde{\sigma} = 3$ for all variables $\rho_{ij}$ and $\rho_{ik}$.    
Therefore $\mathbb{P}(\rho_{ij} > -2.04) = 0.9793, \hspace{0.2cm} \mathbb{P}(\rho_{ik} < 2.04) = 0.9793,$ and
%
%
\begin{equation*}
    \prod_{j=1}^N \varphi_j = 0.9793^{5} = 0.9007.
\end{equation*} 

For each UE, $10^5$ combination samples with different values for each standard Gaussian variable $\rho_{ij}$ and $\rho_{ik}$ were generated. We then calculated the throughput of the respective UE considering the transmission powers, resource allocation, and association given by optimization problem \eqref{min_RGP} for each combination of $\rho_{ij}$ and $\rho_{ik}$. Notice in Figure \ref{fig_stats} that even if around 10\% of the samples do not belong to the uncertainty box, there is no constraint violation for many of those scenarios. For instance consider UE 130: 10.088\% of the samples combination do not belong to the proposed interval, however in just 1.408\% the respective throughput constraint was violated.
\begin{figure*}[t]
\centering
\includegraphics[width=0.3\linewidth]{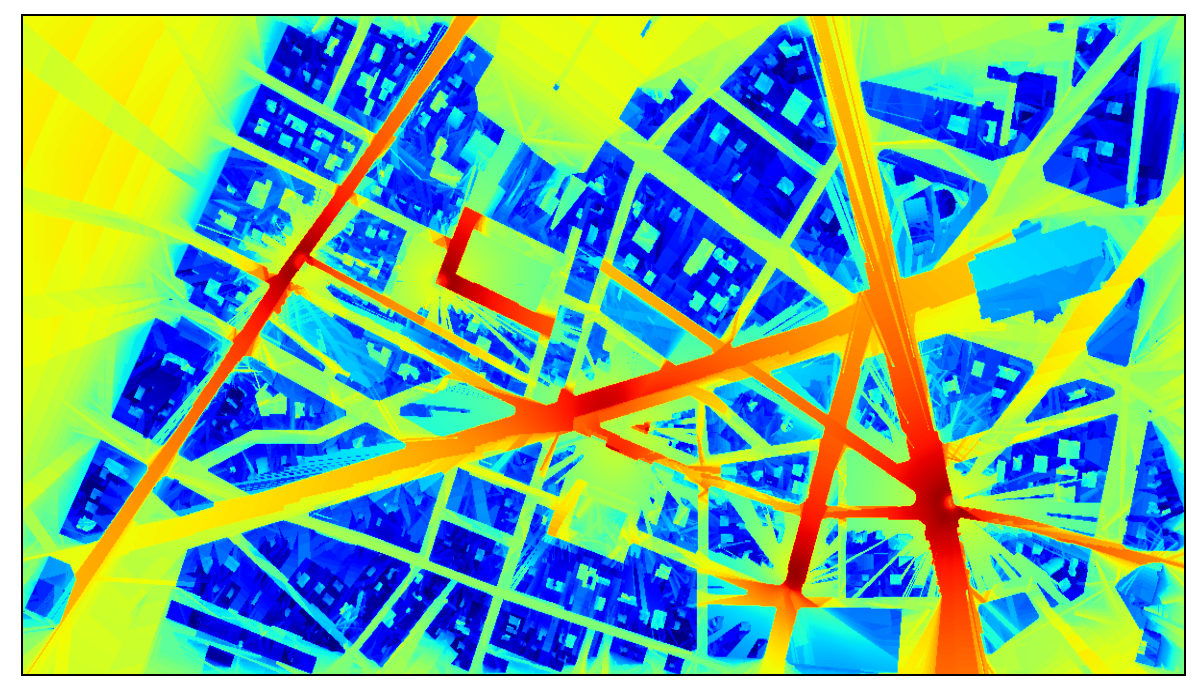}%
\hfil
\includegraphics[width=0.3\linewidth]{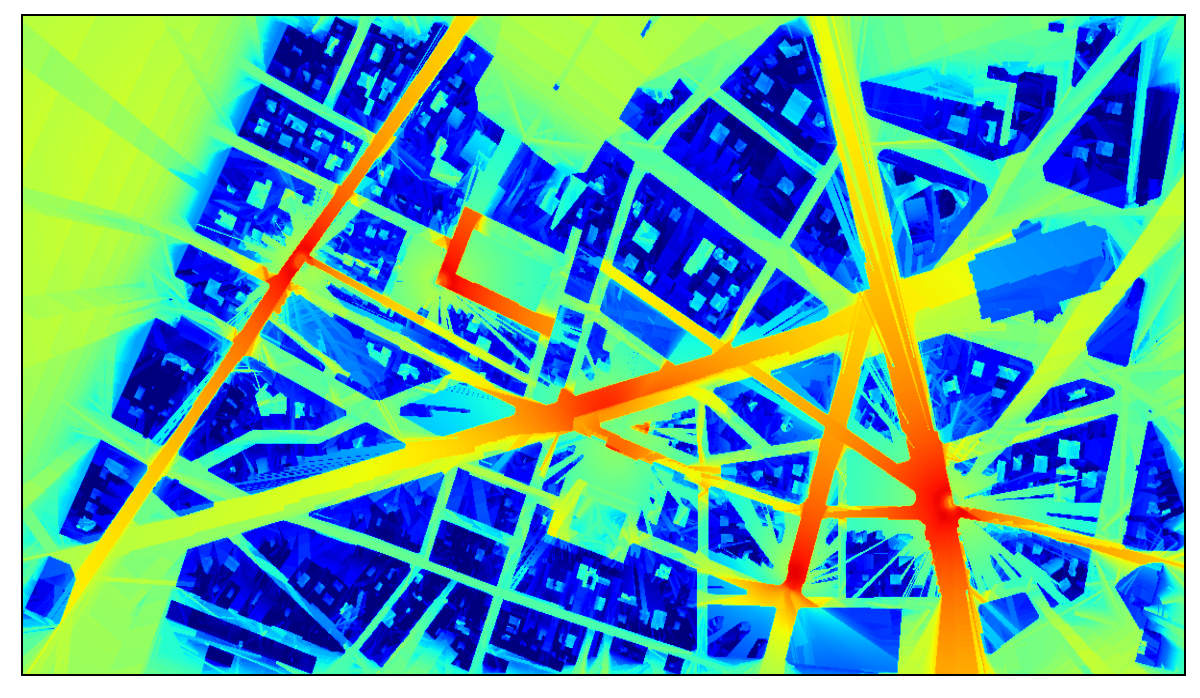}%
\hfil
\includegraphics[width=0.35\linewidth]{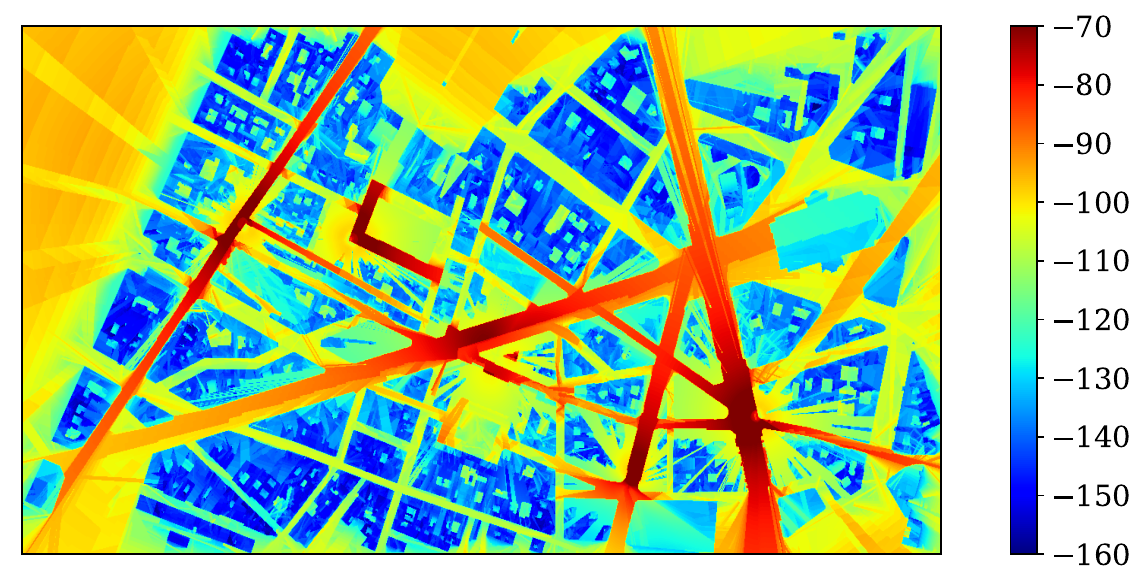} \\
\makebox[0.3\linewidth][c]{\small (a)}%
\hfil
\makebox[0.3\linewidth][c]{\small (b)}%
\hfil
\makebox[0.35\linewidth][c]{\small (c)} \\
\caption{Channel gains (dB) for $\tilde{\sigma} = 3$ and (a) $\tilde{\mu}_{ij}$, (b) $\tilde{\mu}_{ij} - 2 \tilde{\sigma}$, and (c) $\tilde{\mu}_{ij} + 2 \tilde{\sigma}$.}
\label{fig_std_cg}
\end{figure*}

\begin{figure}[ht!]
    \centering
    \includegraphics[width=8.8cm]{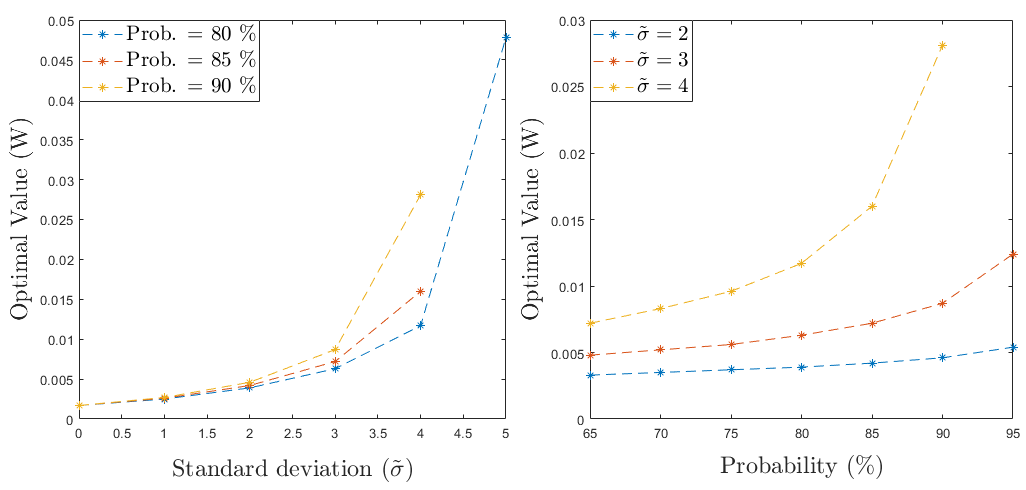}
    \makebox[0.5\linewidth][c]{\small (a)}%
    \hfil
    \makebox[0.5\linewidth][c]{\small (b)}%
    \caption{The optimal value as a function of (a) $\tilde{\sigma}$ and (b) constraint probability.}
    \label{fig_sigma_rho}
\end{figure}
\begin{figure}[ht!]
    \centering
    \includegraphics[width=8.8cm]{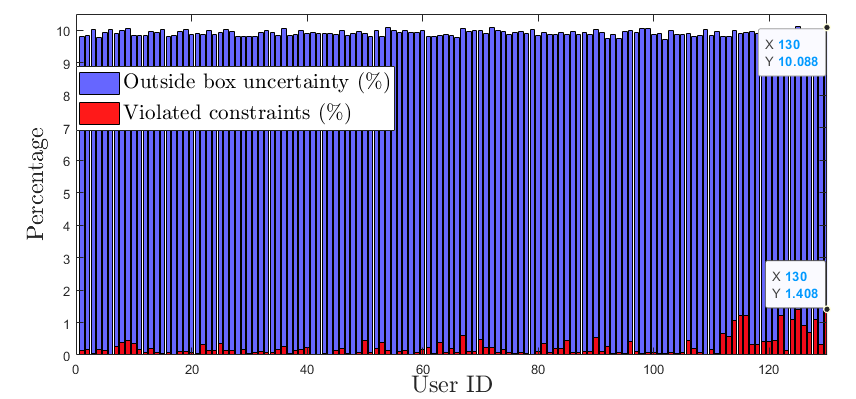}
    \caption{Percentage of channel gains combination outside box of uncertainties $i\in[n]$ and violated constraints.}
    \label{fig_stats}
\end{figure}

\subsection{Comparison}
\label{sec:comp}

In \cite{Ferreira_2024}, comparisons between the \eqref{prob_MIGP} and  \eqref{nonconvex-OFDMA} were performed, where the latter one was solved with 1) successive linearization and gradient descent (GD), 2) difference of convex (DC) functions programming, and 3) with the commercial solver MIDACO (Ant Colony Optimization). In all cases, the GP based optimization problem provides remarkably greater optimal value with substantially smaller solving time. Therefore here we perform comparisons only between the deterministic and stochastic approaches, namely \eqref{prob_MIGP} and \eqref{min_RGP_label} respectively. We solved both problems and tested the optimal solutions in $10^5$ different channel gains combinations (with log-normal distribution and a given $\tilde{\sigma}$ for each user). The percentage of times that the throughput constraints were violated is presented in Table \ref{tab:table_comp}. Despite providing a smaller optimal value, the deterministic approach fails to respect users demands in many more situations, since they were designed considering the $\tilde{\sigma}=0$. 

\begin{table}[ht!]
  \begin{center}
    \caption{Percentage of unsatisfied users' throughput requirements.}
    \label{tab:table_comp}
    \begin{tabular}{c|cc|cc} 
       & \multicolumn{2}{c|}{\textbf{Robust MIGP}} & \multicolumn{2}{c}{\textbf{MIGP}} \\ 
      \hline
      $\tilde{\sigma}$ & Percentage & $\sum_j P_j$ (W) & Percentage & $\sum_j P_j (\tilde{\sigma}=0)$\\ 
      \hline \hline
       2 & 0.29 \% & 0.0046 & 29.40 \% & 0.0017 \\
       3 & 0.24 \% & 0.0087 & 34.93 \% & 0.0017 \\
       4 & 0.12 \% & 0.0281 & 38.40 \% & 0.0017 \\
      \hline
    \end{tabular}
  \end{center}
\end{table}

\subsection{Robust MIGP and different channel gains distribution}

Even if the Robust MIGP assumes log-normal distribution for the channel gains, it is able to respect the chance-constrained optimization problem in different scenarios. Notice in Table \ref{tab:comp_dist} that the proposed approach can easily guarantee that at least 90\% of the users' throughput demands are respected when the channel gains combinations come from a log-normal distribution. This is expected since the problem is formulated under that assumption. If the optimization problem is solved considering log-normal distribution, but the channel gains respect a uniform one, then the 90\% of the users' throughput demands are respected for the cases where the channel gains belong to the following intervals: [$\tilde{\mu}-\tilde{\sigma}, \tilde{\mu}+\tilde{\sigma}$], [$\tilde{\mu}-2\tilde{\sigma}, \tilde{\mu}+2\tilde{\sigma}$], and [$\tilde{\mu}-3\tilde{\sigma}, \tilde{\mu}+3\tilde{\sigma}$]. Considering a Student's t distribution with degree of freedom $dof=2$, the 10\% violation goal was also respected. As $Dof \rightarrow \infty$, the Student's t distribution tends to the Gaussian one (log-normal channel gains, which are Gaussian considering the respective values in dB). 

\begin{table}[ht!]
  \begin{center}
    \caption{Percentage of unsatisfied users' throughput requirements.}
    \label{tab:comp_dist}
    \begin{tabular}{c|c|c|c|c} 
      $\tilde{\sigma}$ & Log-Normal & 
      $\begin{array}{c} \text{Uniform} \\ \text{[$\tilde{\mu}-3\tilde{\sigma},$} \\ \text{$\tilde{\mu}+3\tilde{\sigma}$]} \end{array}$ & 
    $\begin{array}{c} \text{Student's t} \\ \text{Distribution} \\ \text{$Dof$ = 2} \end{array}$ & Optimal Value\\
      \hline \hline
      2 & 0.29\% & 3.86 \%  & 6.52 \%  & 0.0046\\
      3 & 0.24\% & 4.20 \%  & 7.82 \%  & 0.0087\\
      4 & 0.12\% & 3.74 \%  & 8.53 \%  & 0.0281\\
      \hline
    \end{tabular}
  \end{center}
\end{table}

To achieve the 90\% goal even for the case of uniform distribution and [$\tilde{\mu}-4\tilde{\sigma}, \tilde{\mu}+4\tilde{\sigma}$], one could increase the uncertainty box, according to
\begin{equation*}
\begin{split}
[\underline{\rho}_{ij},\overline{\rho}_{ij}] = [-2.25, \infty), \quad 
[\underline{\rho}_{ik},\overline{\rho}_{ik}] = (-\infty,2.25],
\end{split}
\end{equation*}
with the cost of increasing the optimal value, as presented in Table \ref{tab:table_comp_2}.
\begin{table}[ht!]
  \begin{center}
    \caption{Percentage of unsatisfied users' throughput requirements under uniform distribution.}
    \label{tab:table_comp_2}
    \begin{tabular}{c|cc|cc} 
       & \multicolumn{2}{c|}{\textbf{Uncertainty box: 90\%}} & \multicolumn{2}{c}{\textbf{Uncertainty box: 94\%}} \\ 
      \hline
      $\tilde{\sigma}$ & $\begin{array}{c} \text{[$\tilde{\mu}-4\tilde{\sigma},$} \\ \text{$\tilde{\mu}+4\tilde{\sigma}$]} \end{array}$ & $\sum_j P_j$ (W) & $\begin{array}{c} \text{[$\tilde{\mu}-4\tilde{\sigma},$} \\ \text{$\tilde{\mu}+4\tilde{\sigma}$]} \end{array}$ & $\sum_j P_j (\tilde{\sigma}=0)$\\ 
      \hline \hline
       2 & 12.91 \% & 0.0046 & 9.88 \% & 0.0052 \\
       3 & 14.34 \% & 0.0087 & 10.79 \% & 0.0113 \\
       4 & 14.21 \% & 0.0281 & 9.92 \% & 0.1403 \\
      \hline
    \end{tabular}
  \end{center}
\end{table}

\subsection{Robustness with respect to users' traffic demands}

To examine the robustness of our approach with respect to the user traffic demands, we solved the problem considering the original throughput levels given by $r$, $\rho_{ij} = -2.04$, $\rho_{ik} = 2.04$. Subsequently we increased the traffic demands and verify the amount of times that those new constraints were not satisfied (with the optimal solution obtained for the original throughput levels). The results are given below. 

\begin{table}[ht!]
  \begin{center}
    \caption{Percentage of unsatisfied users' throughput requirements.}
    \label{tab:table_rob}
    \begin{tabular}{c|c|ccccc} 
      $\tilde{\sigma}$ & $r$ & 1.1 $r$ & 1.2 $r$ & 1.4 $r$ & 1.5 $r$ & 2.3 $r$\\
      \hline \hline
      2 & 0.29\% & 0.70 \% & 1.73 \% & 7.18 \% & 11.33 \% & 40.00 \%\\
      3 & 0.24\% & 0.43 \% & 0.77 \% & 2.44 \% & 3.96 \%  & 23.09 \%\\
      4 & 0.12\% & 0.19 \% & 0.29 \% & 0.72 \% & 1.13 \%  & 10.20 \%\\
      \hline
    \end{tabular}
  \end{center}
\end{table}

It is clear that for larger values of standard deviation $\tilde{\sigma}$, for instance $\tilde{\sigma} = 4$, the uncertainty box increases leading to an optimal solution that is able to respect the probabilistic constraint even when the users throughput are up to about 130\% higher than the ones used in the optimization problem. From one side, this shows that our approach is able to deal with fast and large variations in the throughput demands; however one may notice that the approaches gets more conservative when the channel gains uncertainty increases.  

\section{Conclusions and Future Works}
\label{sec:conclusion}

The minimization of base stations transmission powers in OFDMA networks, subject to individual UEs throughput constraints is studied in this paper. The channel gains between UE and base stations are considered to be random variables with log-normal distribution. By using change of variables with a suitable power function approximation, the highly non-convex chance constrained optimization problem is formulated as a robust mixed-integer Geometric Program. The proposed approach was tested in a realistic scenario: a neighbourhood from a large European city was reconstructed into a ray tracing propagation solver (considering building height, city layout, BSs location, etc), so the expected values for the channel gains could be estimated. The results show that considering the channel gains as deterministic might lead to considerable amount of constraints violation. For future works, less-conservative estimate for the set of uncertainties will be designed.


\small
\bibliographystyle{ieeetr}
\bibliography{ComLet}
\vfill

\end{document}